\documentclass[11pt]{amsart}
\usepackage{url}
\usepackage[all]{xypic}
\usepackage{mathrsfs}
\input xypic
\usepackage{amsmath,amssymb,amsfonts,amsthm}
\usepackage[top=3.0cm, bottom=3.0cm, left=3.0cm, right=3.0cm]{geometry}
\newtheorem{thm}{Theorem}[section]
\newtheorem{theorem}[thm]{Theorem}
\newtheorem {definition}[thm]{Definition}

\newtheorem{corollary}[thm]{Corollary}
\newtheorem{lemma}[thm]{Lemma}

\newtheorem{proposition}[thm]{Proposition}
\newtheorem{example}[thm]{Example}
\newtheorem{remark}[thm]{Remark}

\newtheorem {conjecture}[thm]{Conjecture}

\begin{document}
\newpage
\thispagestyle{empty}

\baselineskip=16pt
\title{The right classification of univariate power series in positive characteristic}      
\author{Nguyen Hong Duc}
\address{Nguyen Hong Duc
\newline\indent Institute of Mathematics, 18 Hoang Quoc Viet Road, Cau Giay District \newline \indent  10307, Hanoi.} 
\email{nhduc@math.ac.vn}
\address{Universit\"{a}t Kaiserslautern, Fachbereich Mathematik, Erwin-Schr\"{o}dinger-Strasse,  
\newline \indent 67663 Kaiserslautern}
\email{dnguyen@mathematik.uni-kl.de}

\date{\today}                  
\maketitle
\begin{abstract}
While the classification of univariate power series up to coordinate change is trivial in characteristic 0, this classification is very different in positive characteristic. In this note we give a complete classification of univariate power series $f\in K[[x]]$, where $K$ is an algebraically closed field of characteristic $p>0$ by explicit normal forms. We show that the right determinacy of $f$ is completely determined by its support. Moreover we prove that the right modality of $f$ is equal to the integer part of $\mu/p$, where $\mu$ is the Milnor number of $f$. As a consequence we prove in this case that the modality is equal to the proper modality, which is the dimension of the $\mu$-constant stratum in an algebraic representative of the semiuniversal deformation with trivial section.
\end{abstract}
\section{Introduction}
In \cite{Arn72} V.I. Arnol'd introduced the ``modality", or the number of moduli, for real and complex hypersurface singularities and he classified singularities with modality smaller than or equal to 2. In oder to generalize the notion of modality to the algebraic setting, the author and Greuel in \cite{GN13} introduced the modality for algebraic group actions and applied it to high jet spaces. 

Let the algebraic group $G$ act on the variety $X$. Then there exists a {\em Rosenlicht stratification} $\{(X_i,p_i), i=1,\ldots, s\}$ of $X$ w.r.t. $G$, i.e. the $X_i$ is a locally closed $G$-invariant subset of $X$, $X=\cup_{i=1}^sX_i$ and the $p_i:X_i\to X_i/G$ a geometric quotient. For each open subset $U\subset X$ we define
$$G\text{-}\mathrm{mod}(U):=\max_{1\leq i\leq s}\{\dim \big (p_i(U\cap X_i)\big)\},$$
and for $x\in X$ we call 
$$G\text{-}\mathrm{mod}(x):=\min \{G\text{-}\mathrm{mod}(U)\ |\ U \text{ a neighbourhood of x}\}$$
the {\em $G$-modality} of $x$.

Let $K$ be an algebraically closed field of characteristic $p\geq 0$, let $K[[{\bf x}]]=K[[x_1,\ldots,x_n]]$ be the formal power series ring and let the right group, $\mathcal R:=Aut(K[[{\bf x]}])$, act on $K[[{\bf x}]]$ by $(\Phi,f) \mapsto \Phi(f)$. Two elements $f,g\in K[[{\bf x}]]$ are called {\em right equivalent}, $f\sim_r g$, if they belong to the same $\mathcal{R}$-orbit, or equivalently, there exists a coordinate change $\Phi\in Aut(K[[{\bf x]}])$ such that $g=\Phi(f)$.

Let $f\in \langle {\bf x}\rangle\subset K[[{\bf x}]]$ and let $\mu(f):=\dim K[[{\bf x}]]/\langle f_{x_1}, \ldots, f_{x_n}\rangle$ be its Milnor number. We call $f$ {\em isolated} if $\mu(f)<\infty$. By \cite[Thm. 5]{BGM12}, $f$ is isolated if and only if it is finitely right determined, i.e. $f$ is right $k$-determined for some $k$. Here $f$ is {\em right $k$-determined} if each $g\in K[[{\bf x}]]$ s.t. $j^k g =j^k f$, is right equivalent to $f$, where $j^k f$ denotes the {\em $k$-jet} of $f$ in the {\em $k$-th jet space} $J_k:=\langle {\bf x}\rangle/\langle {\bf x}\rangle^{k+1}$. The minimum of such $k$ is called the {\em right determinacy} of $f$. For each isolated $f$, the {\em right modality} of $f$, $\mathcal R\text{-}\mathrm{mod}(f)$, is defined to be the $\mathcal R_k\text{-}\mathrm{modality}$ of $j^k f$ in $J_k$ with $k\geq 2\mu(f)$ and $\mathcal R_k$ the $k$-jet of $\mathcal R$. Notice that if $f$ is right equivalent to $g$ then $\mathcal R\text{-}\mathrm{mod}(f)=\mathcal R\text{-}\mathrm{mod}(g)$ (cf. \cite[Prop. A.4]{GN13}).

In Section 2, we show that the right determinacy of an isolated univariate formal power series $f$ is equal to $d(f)$, which is defined by a concrete formula determined by the support of $f$ (Definition \ref{def2.1}, Proposition \ref{pro2.2}). Moreover we give an explicit normal form for any (not necessary isolated) univariate power series $f$ w.r.t. right equivalence (Theorem \ref{thm2.1}). We prove in Section 3 that the right modality of an isolated series $f$ is equal to the integer part of $\mu(f)/p$ (Theorem \ref{thm3.1}). As a consequence we show that the right modality is equal to the dimension of the $\mu$-constant stratum in an algebraic representative of the semiuniversal deformation with trivial section (Corollary \ref{coro3.1}).  

\subsection*{Acknowledgement} 
We would like to thank the referees for their careful reading of the manuscript and helpful comments which improved the presentation of this paper. The result of this article is part of my thesis \cite{Ng13} under the supervision of Professor Gert-Martin Greuel at the Technische Universit\"at Kaiserslautern. I am grateful to him for many valuable suggestions. 
This author's research was partially supported by Vietnam National Foundation for Science and Technology Development(NAFOSTED) grant 101.04-2014.23, and DAAD (Germany).

\section{Normal forms of univariate power series}
Let $f=\sum_{n\geq 0} c_nx^n\in K[[x]]$ be a univariate power series, let $\mathrm{supp}(f):=\{n\geq 0\ |\ c_n\neq 0\}$ be the {\em support} of $f$ and $\mathrm{mt}(f):=\min\{n\ |\ n\in \mathrm{supp}(f)\}$ the {\em multiplicity} of $f$. If $\mathrm{char}(K)=0$ and if $\varphi(x)=a_1x+a_2x^2+\ldots, a_1\neq 0,$ is a coordinate change, then the coefficients $a_i$ of $\varphi$ can be determined inductively from the equation $f(x)=c_0+(\varphi(x))^{\mathrm{mt}(g)}$ with $g(x):=f-c_0$. Hence $f$ is right equivalent to $c_0+x^{\mathrm{mt}(g)}$.

In the following we investigate $f\in K[[x]]$ with $\mathrm{char}(K)=p>0$. The aim of this section is to give a normal form of $f$. It turns out that it depends in a complicated way on the divisibility relation between $p$ and the support of $f$. To describe this relation we make the following definition, where later on $\Delta$ will be $\mathrm{supp}(f)$. 
\begin{definition}\label{def2.1}{\rm
For each $n\in \mathbb N$ and each non-empty subset $\Delta\subset \mathbb N\setminus \{0\}$, we define
\begin{itemize} 
\item[(a)] $m:=m(\Delta):=\min\{n\ |\ n\in \Delta\}$.
\item[(b)] $e:=e(\Delta):=\min \{e(n)\ |\ n\in \Delta\},$ where $e(n):=\max\{i\ |\ p^i \text{ divides } n\}$.
\item[(c)]  $q:=q(\Delta):=\min \{n\in \Delta\ |\ e(n)=e\}.$
\item[(d)] $k:=k(\Delta):=1$ and $e_0(\Delta):=e+1$ if $e(m)=e$ (i.e. $m=q$), otherwise, $$k:=k(\Delta):=\max\{k_{\Delta}(n)\ |\ m\leq n <q, n\in \Delta\},$$ 
where
$$k_{\Delta}(n):=\left\lceil \frac{q-n}{p^{e(n)}-p^{e}}\right\rceil \text{ denotes the ceiling of  }\frac{q-n}{p^{e(n)}-p^{e}}$$
and 
$$e_0:=e_0(\Delta):=\min\{e(n)\ |\ m\leq n<q, n\in \Delta\}.$$
\item[(e)] $d:=d(\Delta):=q+p^{e}(k-1)$.
\item[(f)] $\bar\Lambda(\Delta)=\emptyset$ if $e(m)=e$, otherwise, 
$$\bar\Lambda(\Delta):=\{n\in \mathbb N \ |\ m<n\leq d, e_0\leq e(n) \}\cup \{q\}.$$
\item[(g)] If $e(m)>e$ (i.e. $m<q$) we define

$\Delta_0:= \{n\in \Delta\ |\ n<q\},\ q_0:=q(\Delta_0),\ d_0:=d(\Delta_0), \ \bar d_0:=\min\{d,d_0\}$,

$\Lambda_0(\Delta):=\emptyset$ if $e(m)=e_0$,

$\Lambda_0(\Delta):=\left\{ n\in \mathbb N\ |\  m<n\leq \bar d_0,  e_0<e(n)\right\}\cup\{q_0\}$ if $e(m)>e_0$, and

$\Lambda_1(\Delta):=\{n\in \mathbb N\ |\ q\leq n\leq d,e\leq e(n)<e_0\}$.

\item[(h)] If $e(m)=e$ then $\Lambda(\Delta):=\emptyset$, otherwise,
$$\Lambda(\Delta):=\Lambda_0(\Delta)\cup \Lambda_1(\Delta).$$
\end{itemize}
}\end{definition}
\begin{remark}\label{rm2.0}{\rm
If $f\in K[[x]]$ with $\mu(f)<\infty$ and $\Delta=\mathrm{supp}(f)$ then
\begin{itemize}
\item[(a)] $m(\Delta)=\mathrm{mt}(f)$, the multiplicity (or, the order) of $f$. 
\item[(b)] $q(\Delta)=\mu(f)+1$, the first exponent in the expansion of $f$ which is not divisible by $p$.
\item[(c)] $k_{\Delta}(n)$ is the minimum of $l$ for which  
$$\mathrm{mt}\left(\varphi(x^n)-x^n\right)\geq \mathrm{mt}\left(\varphi(x^q)-x^q\right)= {q+l}$$  
with $q:=q(\Delta)$ and $\varphi=x+u_{l+1}x^{l+1}+\text{terms of higher order}, u_{l+1}\neq 0$, a coordinate change.

Indeed,
\begin{eqnarray*}
\varphi (x^n)&=& \left(x+u_{l+1}x^{l+1}+\ldots\right)^n\\
&=&\left[\left( x+u_{l+1}x^{l+1}+\ldots\right)^{n/{p^{e(n)}}} \right]^{p^{e(n)}}\\
&=&\left[x^{n/{p^{e(n)}}}+({n/{p^{e(n)}}})\cdot u_{l+1} x^{n/{p^{e(n)}}+l}+\ldots \right]^{p^{e(n)}}\\
&=&x^n+({n/{p^{e(n)}}})^{p^{e(n)}}u_{l+1}^{p^{e(n)}}x^{n+l{p^{e(n)}}}+\ldots .
\end{eqnarray*}
It yields that
$$\mathrm{mt}\left(\varphi(x^n)-x^n\right)\geq {q+l}\Leftrightarrow l\geq \frac{q-n}{p^{e(n)}-1}.$$
This proves the claim.
\item[(d)] $k(\Delta)$ is then the minimum of $l$ for which 
$$\varphi(f)=f\mod x^{q+l}$$ 
with $q=q(\Delta)$ and a coordinate change $\varphi$ as above. This is used to show that:
\item[(e)] $d(\Delta)$ is the right determinacy of $f$, cf. Proposition \ref{pro2.2}.  
\end{itemize}
}\end{remark}
\begin{remark}\label{rm2.1}{\rm
The following facts (a)-(e) are immediate consequences of the definition. Property (f) follows from elementary calculations.
\begin{itemize}
\item[(a)] $e(\Delta)<e_0(\Delta)$, $k(\Delta)>0$.
\item[(b)] If $q(\Delta)=q(\Delta')=:q$ and $\Delta\cap \mathbb N_{<q}=\Delta'\cap \mathbb N_{<q}$, then $d(\Delta)=d(\Delta')$ and $\Lambda(\Delta)\equiv \Lambda(\Delta')$. That is, $q(\Delta)$ is the ``determinacy'' of $\Lambda(\Delta)$.
\item[(c)] If $p$ does not divide $m(\Delta)$, then
\begin{itemize}
\item[1.] $e(\Delta)=e(m(\Delta))=0$ and $q(\Delta)=m(\Delta)$.  
\item[2.] $k(\Delta)=1$ and $d(\Delta)=m(\Delta)$.
\end{itemize}
\item[(d)] If $e(m(\Delta))=e(\Delta)$, then
\begin{itemize}
\item[1.] $q(\Delta)=m(\Delta)$.  
\item[2.] $k(\Delta)=1$ and $d(\Delta)=m(\Delta)$.
\end{itemize}
\item[(e)] If $n+lp^{e(n)}\leq d(\Delta)$ for some $l$ and some $n\in \Delta$, then $l\leq k(\Delta)$.
\item[(f)] If $k(\Delta)=k_{\Delta}(n)$, then 
$$k(\Delta)-1+\frac{n}{p^{e(n)}}=\left\lfloor\frac{d(\Delta)}{p^{e(n)}}\right\rfloor,$$
where $\left\lfloor\frac{d(\Delta)}{p^{e(n)}}\right\rfloor$ denotes the floor (or, integer part) of $\frac{d(\Delta)}{p^{e(n)}}$.

In fact, one has, by denoting $e:=e(\Delta),\ q:=q(\Delta),\ k:=k(\Delta),\ d:=d(\Delta)$, that 
\begin{eqnarray*}
\frac{d}{p^{e(n)}}-\left( k-1+\frac{n}{p^{e(n)}}\right) &=& \frac{q+p^e(k-1)}{p^{e(n)}}-\left( k-1+\frac{n}{p^{e(n)}}\right)\\
&=&\frac{p^{e(n)}-p^e}{p^{e(n)}}\cdot \left( \frac{q-n}{p^{e(n)}-p^e}-k+1\right).
\end{eqnarray*}
Then $$0<\frac{d}{p^{e(n)}}-\left(k+\frac{n}{p^{e(n)}}-1\right)<1$$
since $k=\left\lceil\frac{q-n}{p^{e(n)}-p^e}\right\rceil$. This gives us the formula.

\end{itemize}
}\end{remark}
\begin{example}\label{ex2.1}{\rm
Let $p=\mathrm{char}(K)=2$, let
$$f=x^8+x^{36}+x^{37}+\text{terms of higher order in } K[[x]],$$
and let 
$$\Delta:=\mathrm{supp}(f)=\{8,36,37,\ldots\}.$$
Then 
$$e=0,q=37,k=k_{\Delta}(8)=5, d=41 ,$$
$$e_0=2,q_0=36,d_0=60,\bar d_0=d=41.$$
and
$$\Lambda(f)=\{16,24,32,36,37,38,39,40,41\},$$
$$\sharp\Lambda(f)=9=\left\lfloor\frac{q}{p^{e_0}}\right\rfloor-\left\lfloor\frac{m}{p^{e_0}}\right\rfloor+2.$$
}\end{example}
The following proposition is the first key step in the classification.
\begin{proposition}\label{pro2.1}
With the notions as in Definition \ref{def2.1}, assume that $e(\Delta)=0$. Then 
$$\sharp \Lambda(\Delta)\leq \left\lfloor \frac{q}{p}\right\rfloor-\frac{m}{p}+1.$$
More precisely,
\begin{itemize}
\item[(i)] If $e(m)<e_0$ then $\sharp \Lambda(\Delta)=0$.
\item[(ii)] If $e(m)=e_0$ then $\sharp \Lambda(\Delta)= \left\lfloor \frac{q}{p^{e_0}}\right\rfloor-\frac{m}{p^{e_0}}+1.$
\item[(iii)] If $e(m)>e_0$ and 
\begin{itemize}
\item[(1)] if $p>2$ then $\sharp \Lambda(\Delta)\leq \left\lfloor \frac{q}{p^{e_0}}\right\rfloor-\frac{m}{p^{e_0}}+1$;
\item[(2)] if $p=2$ then $\sharp \Lambda(\Delta)\leq \left\lfloor \frac{q}{p^{e_0}}\right\rfloor-\frac{m}{p^{e_0}}+2.$
\end{itemize} 
\end{itemize} 
\end{proposition}
\begin{proof} (i) It is easy to see that, $e(m)<e_0$ if and only if $e(m)=e$ and then $\Lambda(\Delta)=\emptyset$.

(ii) Since $e(m)=e_0$, $\Lambda_0(\Delta)=\emptyset$ and $k_{\Delta}(m)=k$. Then
$$\Lambda(\Delta)=\Lambda_1(\Delta)=\left\{n\in \mathbb N\ |\ q \leq  n\leq d, e(n)<e_0\right\}$$
and hence
$$\sharp \Lambda(\Delta)=k-\left(\left\lfloor\frac{d}{p^{e_0}}\right\rfloor-\left\lfloor\frac{q}{p^{e_0}}\right\rfloor\right)=\left\lfloor \frac{q}{p^{e_0}}\right\rfloor-\frac{m}{p^{e_0}}+1$$
since $k-1+\frac{m}{p^{e(m)}}=\left\lfloor\frac{d}{p^{e(m)}}\right\rfloor$ due to Remark \ref{rm2.1}(f).

(iii) Since $e(m)>e_0$ one has
$$k(\Delta_0)-1=\left\lceil\frac{q_0-n}{p^{e(n)}-p^{e_0}}\right\rceil-1<\frac{q_0-m}{p^{e_0+1}-p^{e_0}}$$
for some $n\in \Delta_0,\ e(n)>e_0$, and 
$$\Lambda_0(\Delta)=\left\{ n'\in \mathbb N\ |\  m<n'\leq \bar d_0, e(n')>e_0\right\}\cup \{q_0\},$$
$$\Lambda_1(\Delta)=\left\{n'\in \mathbb N\ |\ q \leq  n'\leq d, e(n')<e_0\right\}.$$
This implies that
$$\sharp \Lambda_0(\Delta)=\left\lfloor\frac{\bar d_0}{p^{e_0+1}}\right\rfloor-\frac{m}{p^{e_0+1}}+1$$
and
\begin{eqnarray*}
\sharp \Lambda_1(\Delta)&=&(d-q+1)-\left( \left\lfloor\frac{d}{p^{e_0}}\right\rfloor-\left\lfloor\frac{q}{p^{e_0}}\right\rfloor\right)\\
&=&k-\left( \left\lfloor\frac{d}{p^{e_0}}\right\rfloor-\left\lfloor\frac{q}{p^{e_0}}\right\rfloor\right).
\end{eqnarray*}
We consider the following cases:
$$$$
{\bf Case 1:} $k_{\Delta}(q_0)=k$.

Then $k-1+\frac{q_0}{p^{e_0}}=\left\lfloor\frac{d}{p^{e_0}}\right\rfloor$ by Remark \ref{rm2.1}(f). We obtain 
\begin{eqnarray*}
\sharp \Lambda(\Delta)&=&\sharp \Lambda_0(\Delta)+\sharp \Lambda_1(\Delta)=\left\lfloor\frac{q}{p^{e_0}}\right\rfloor-\left(\frac{q_0}{p^{e_0}}-\left\lfloor\frac{\bar d_0}{p^{e_0+1}}\right\rfloor+\frac{m}{p^{e_0+1}}-2\right)\\
&\leq&\left\lfloor\frac{q}{p^{e_0}}\right\rfloor-\left(\frac{q_0}{p^{e_0}}-\left\lfloor\frac{d_0}{p^{e_0+1}}\right\rfloor+\frac{m}{p^{e_0+1}}-2\right)\\
&\leq&\left\lfloor\frac{q}{p^{e_0}}\right\rfloor-\left(\frac{q_0}{p^{e_0}}-\frac{q_0+\left(k(\Delta_0)-1\right)p^{e_0}}{p^{e_0+1}}+\frac{m}{p^{e_0+1}}-2\right)\\
&< &\left\lfloor\frac{q}{p^{e_0}}\right\rfloor-\left(\frac{(p^2-2p)q_0+m}{p^{e_0+2}-p^{e_0+1}}+\frac{m}{p^{e_0+1}}-2\right)\\
&\leq&\left\lfloor\frac{q}{p^{e_0}}\right\rfloor-\left(\frac{m}{p^{e_0}}-2\right),
\end{eqnarray*}
due to $k(\Delta_0)-1<\frac{q_0-m}{p^{e_0+1}-p^{e_0}},$ respectively $q_0>m$. Hence
$$\sharp \Lambda(\Delta)\leq \left\lfloor \frac{q}{p^{e_0}}\right\rfloor-\frac{m}{p^{e_0}}+1.$$ 
{\bf Case 2:} $k_{\Delta}(q_0)<k$.

Then
$$k=\left\lceil\frac{q-n}{p^{e(n)}-1}\right\rceil<\frac{q-m}{p^{e_0+1}-1}+1$$
for some $n\in \Delta_0,\ e(n)>e_0$. It yields that
$$d=q+k-1>(k-1)p^{e_0+1}+m$$
and hence 
\begin{eqnarray*}
\sharp \Lambda(\Delta)&=&\left\lfloor\frac{q}{p^{e_0}}\right\rfloor-\left(\left\lfloor\frac{d}{p^{e_0}}\right\rfloor-\left\lfloor\frac{\bar d_0}{p^{e_0+1}}\right\rfloor+\frac{m}{p^{e_0+1}}-k-1\right)\\
&\leq &\left\lfloor\frac{q}{p^{e_0}}\right\rfloor-\left(\left\lfloor\frac{d}{p^{e_0}}\right\rfloor-\left\lfloor\frac{d}{p^{e_0+1}}\right\rfloor+\frac{m}{p^{e_0+1}}-k-1\right)\\
&\leq &\left\lfloor\frac{q}{p^{e_0}}\right\rfloor-\left(\left\lfloor\frac{(p-1)d}{p^{e_0+1}}\right\rfloor+\frac{m}{p^{e_0+1}}-k-1\right)\\
&\leq &\left\lfloor\frac{q}{p^{e_0}}\right\rfloor-\left((p-1)(k-1)+\frac{m}{p^{e_0}}-k-1\right)\\
&=&\left\lfloor\frac{q}{p^{e_0}}\right\rfloor-\frac{m}{p^{e_0}}+2-(p-2)(k-1).
\end{eqnarray*}
This completes the proposition.
\end{proof}
Note that if $f\in K[[x]]$ and $\mathrm{mt}(f)=0$ then $\mathrm{mt}(f-f(0))>0$. Applying the results from $\mathrm{mt}(f)>0$ to $f-f(0)$ we obtain that $f\sim_r f(0)+g$, where $g$ is a normal form of $f-f(0)$ (cf. Theorem \ref{thm2.1}). From now on we assume that $\mathrm{mt}(f)>0$. We denote, by using notations as in Definition \ref{def2.1} for $\Delta=\mathrm{supp}(f)$,
$$e(f):=e(\Delta),\ q(f):=q(\Delta),\ k(f):=k(\Delta),\ d(f):=d(\Delta)$$ and
$$\bar\Lambda(f):=\bar\Lambda(\Delta),\ \Lambda(f):=\Lambda(\Delta).$$
\begin{remark}\label{rm2.2}
{\rm\begin{itemize}
\item[(a)] The above numbers $\mathrm{mt},e,q,k,d$ and the sets $\Lambda$ and $\bar \Lambda$ are invariant w.r.t. right equivalence.
\item[(b)] Let $f=\sum_{n\geq 1} c_n x^n\in K[[x]]$ and let  
$$\bar f(x)=\sum_{n\geq m(f)}c_n x^{n/p^{e(f)}}.$$
Then $\bar f\in K[[x]]$, $f(x)=\bar f(x^{p^{e(f)}})$ and $e(\bar f)=0$. Moreover, 
$$k(f)=k(\bar f),\sharp \Lambda(f)=\sharp \Lambda(\bar f),\ \sharp \bar\Lambda(f)=\sharp \bar\Lambda(\bar f)$$
and if $\zeta(f)$ denotes one of $\mathrm{mt}(f),e(f),q(f),d(f)$ then
$$\zeta(f)=p^{e(f)} \zeta(\bar f).$$
\item[(c)] Note that $\mu(f)<\infty$ if and only if $e(f)=0$ and then $q(f)=\mu(f)+1$. By \cite[Thm. 2.1]{BGM12} $f$ is then right $\left(2q(f)-\mathrm{mt}(f)\right)$-determined. In Proposition \ref{pro2.2}  we will show that $d(f)$ is the right determinacy of $f$.
\end{itemize}}
\end{remark}
\begin{lemma}\label{lm2.1}
If $e(\mathrm{mt}(f))=e(f)$ then $f\sim_r x^{\mathrm{mt}(f)}$.
\end{lemma}
\begin{proof}
By Remark \ref{rm2.2}, there exists $\bar f\in K[[x]]$ such that $f(x)=\bar f(x^{p^{e(f)}})$ and $e(\bar f)=0$. This implies that $\mu(\bar f)=q(\bar f)-1$ and then $\mu(\bar f)=\mathrm{mt}(\bar f)-1$ since $e(\mathrm{mt}(f))=e(f)$. It follows from \cite[Thm. 2.1]{BGM12} that $\bar f$ is right $(\mathrm{mt}(\bar f)+1)$-determined. That is,
$$\bar f\sim_r c_m x^{\mathrm{mt}(\bar f)}\sim_r x^{\mathrm{mt}(\bar f)}$$ 
 and hence $f\sim_r x^{\mathrm{mt}(f)}$ with the same coordinate change.

In fact, in this case an inductive proof as in the case of characteristic 0 works.
\end{proof}
The next proposition is the second key step in the classification.
\begin{proposition}\label{pro2.2}
With $f$ and $d(f)$ as above, assume that $\mu(f)<\infty$ then $d(f)$ is exactly the right determinacy of $f$.
\end{proposition}
\begin{proof}
We may assume that $e(\mathrm{mt}(f))>e(f)$ since the case $e(\mathrm{mt}(f))=e(f)$ follows from Lemma \ref{lm2.1}. Let us denote $\Delta:=\mathrm{supp}(f)$ and use the notions as in Definition \ref{def2.1}. 
$$$$
{\em Step 1:} Let us show that if $g\in K[[x]]$ with $j^d(f)=j^d(g)$ and $d:=d(f)$ then $f\sim_r g$. 

By Remark \ref{rm2.1}(b), $d(g)=d(f)=d$ since 
$$\mathrm{supp}(f)\cap \{n\in \mathbb N\ |\ n\leq q\}=\mathrm{supp}(g)\cap \{n\in \mathbb N\ |\ n\leq q\}.$$
It suffices to show that
$$f\sim_r f_0:=j^d(f).$$
Indeed, we write
$$f=f_0+f_1\text{ with }\mathrm{mt}(f_1)\geq d+1.$$
and assume without loss of generality, that 
$$f_1=b_{q+l} x^{q+l} +\text{ terms of higher order, with } b_{q+l}\neq 0.$$
Then the coordinate change
$\varphi_1(x)=x+u_{l+1} x^{l+1} $ with $u_{l+1}$ a root of the following non-constant polynomial:
$$qc_q X+\sum_{\frac{q-n}{p^{e(n)}-1}=l} (n/p^{e(n)})^{p^{e(n)}} c_n X^{p^{e(n)}}+b_{q+l}=0 $$
is sufficient to increase the multiplicity of $f_1$ and does not change $f_0$ by Remark \ref{rm2.0}(d). We thus finish by induction.
$$$$
{\em Step 2:} We now show that $f$ is not right $(d-1)$--determined.

For this we need the following\\
{\bf Claim:} $f\sim_r g$ if and only if $j^dg\in \mathcal{R}_{k}\cdot j^df$, where
$$\mathcal{R}_{k}:=\{\psi=u_0x+u_1x^2+\ldots+u_{k-1}x^{k}\ |\ u_0\neq 0\}\subset \mathcal{R}$$
and it acts on the jet space $J_d$ by $(\psi,j^dh)\mapsto j^d(\psi(j^dh))$.\\
{\em Proof of the claim.} The ``if''-statement follows easily from the first step. We assume that $f\sim_r g$, i.e. $g=\varphi(f)$ with
$$\varphi=u_0x+u_1x^2+\ldots, u_0\neq 0.$$ 
Setting
$$\psi:=u_0x+u_1x^2+\ldots+u_{k-1}x^k$$
and $\varphi_1:=\varphi\circ \psi^{-1}$ we obtain that $\varphi=\varphi_1\circ \psi$ and that
$$\varphi_1=x+a_{k}x^{k+1}+\text{ terms of higher order}.$$
Note that $k=k(f)=k(\psi(f))$ due to Remark \ref{rm2.2}(a). It follows from Remark \ref{rm2.0}(d) that
$$j^d\left(\varphi_1(\psi(f))\right)=j^d(\psi(f)).$$
Hence
$$j^dg=j^d\varphi(f)=j^d\left(\varphi_1(\psi(f))\right)=j^d(\psi(f))=j^d(\psi(j^df)).$$
This completes the claim.

We write, for new indeterminates $u_0,\ldots,u_{k-1},t$,  
$$f+tx^d-\psi(j^df)=\sum_{i=m}^{d} b_i(u_0,\ldots,u_{k-1},t) x^i$$
with $\psi:=u_0x+u_1x^2+\ldots+u_{k-1}x^k$ and $b_i\in K[u_0,\ldots,u_{k-1},t]$, and define 
$$V:=Z(b_{m},\ldots, b_d):=\{\left(u_1,\ldots,u_{k-1},t\right)\in \mathbb A^{k}\ |\ b_i(u_0,\ldots,u_{k-1},t)=0\}$$ with the structure sheaf $\mathcal O_V$ and its algebra of global section
$$\mathcal O_V(V)=K[u_0,\ldots,u_{k-1},t]/\langle b_{m},\ldots, b_d\rangle.$$

We prove the second step by contradiction. Suppose the assertion were false. Then for all $t\in K$, $f$ would be right equivalent to $f+tx^{d}$, equivalently, $j^df+tx^{d}\in \mathcal{R}_{k}\cdot j^df$ for all $t$ due to the above claim. This implies that the map $p$ defined by
\begin{displaymath}
\begin{array}{ccccc}
p&:&V&\to & \mathbb A^1\\
&&(u_0,\ldots,u_{k-1},t)&\mapsto & t 
\end{array}
\end{displaymath}
is surjective. It yields that $\dim V\geq 1$. We may assume without loss of generality that $\dim_O V\geq 1$, where $O=(1,0,\ldots,0)\in V$ and $\dim_O V$ denotes the maximal dimension of irreducible components of $V$ containing $O$. Since $\mathcal O_{V,O}\subset R:=K[[u'_0,u_1,\ldots,u_{k-1},t]]/\langle b_{m},\ldots, b_d\rangle$ with $u'_0=u_0-1$,
$$\dim R\geq \dim\mathcal O_{V,O}=\dim_O V\geq 1.$$
By the Curve Selection Lemma, there exists a non-constant $K$--algebra homomorphism 
\begin{eqnarray*}
\phi:K[[u'_0,u_1,\ldots,u_{k-1},t]]&\to & K[[\tau]]\\
u'_0&\mapsto & u'_0(\tau)\\
u_i&\mapsto & u_i(\tau)\\
t&\mapsto & t(\tau)
\end{eqnarray*}
such that 
$$b_i\left(1+u'_0(\tau),u_1(\tau),\ldots,u_{k-1}(\tau),t(\tau)\right)=0\ \text{ for all } i=m,\ldots,d.$$ Since $b_m=c_m(u_0^m-1)$, it follows that
$$\left(1+u'_0(\tau)\right)^m-1=0$$
and therefore $u'_0(\tau)=0$. Notice that, the series $u_i(\tau), i=1,\ldots,k-1$ could not be all equal to zero since $\phi\neq 0$ and since
$$b_{d}(1, u_1,\ldots,u_{k-1},t)=qc_q u_{k-1}+t+b'_d(u_1,\ldots,u_{k-1}), \text{ with } \mathrm{mt}(b'_d)\geq 2.$$
We set
$$l:=\min\{j\ |\ u_j(\tau)\neq 0\},$$
$$L:=\min \{n+lp^{e(n)}\ |\ n\in \Delta\}$$
and
$$I:=\{n\in \Delta\ |\ L=n+lp^{e(n)}\}.$$
By Remark \ref{rm2.0} we can conclude that $m< L< d$ and that
$$\psi(f)-f=\sum_{n\in I}\left(n/p^{e(n)}\right)^{p^{e(n)}} c_n u_l(\tau)^{p^{e(n)}}x^L+\text{ terms of higher order}$$
where $$\psi=x+u_{l}(\tau)x^{l+1}+\ldots+u_{k-1}(\tau)x^k.$$
It follows that 
$$b_L\left(1,u_1(\tau),\ldots,u_{k-1}(\tau),t(\tau)\right)=\sum_{n\in I}\left(n/p^{e(n)}\right)^{p^{e(n)}} c_n u_l(\tau)^{p^{e(n)}}\neq 0,$$
which is a contradiction. This proves the second step.
\end{proof}
In Corollary \ref{coro2.1}, Lemma \ref{lm2.2} and Theorem \ref{thm2.1} below we do not assume that $f$ is an isolated singularity, i.e. $\mu(f)$ may be infinite or, equivalently, $e(f)$ may be bigger than $0$.
\begin{corollary}\label{coro2.1}
Let $f\in K[[x]]$ and $d=d(f)$. Let $g\in K[[x]]$ be such that $e(f)=e(g)$ and $j^d(f)=j^d(g)$. Then $f\sim_r g$. 

We have in particular that $f\sim_r j^d(f)$.
\end{corollary}
\begin{proof}
By Proposition \ref{pro2.2}, it suffices to prove the corollary for the case that $e:=e(f)=e(g)>0$. Taking $\bar f\in K[[x]]$ and $\bar g\in K[[x]]$ such that $f(x)=\bar f(x^{p^e}),\ g(x)=\bar g(x^{p^e})$ as in Remark \ref{rm2.2} we have
$$e(\bar f)=e(\bar g)=0,\ \bar d:=d(\bar f)=d/p^e.$$
Since $j^d(f)=j^d(g)$, $j^{\bar d}(\bar f)=j^{\bar d}(\bar g)$ and hence $\bar f\sim_r \bar g$ according to Proposition \ref{pro2.2}. This implies $f\sim_r g$ with the same coordinate change.
\end{proof}
\begin{lemma}\label{lm2.2}
With $f$, $\mathrm{mt}(f)$ and $\bar\Lambda(f)$ as above, we have   
$$f\sim_r x^{\mathrm{mt}(f)}+\sum_{n\in \bar\Lambda(f)}\lambda_n x^n,$$
for suitable $\lambda_n\in K$. 
\end{lemma}
\begin{proof}
We decompose $f=f_0+f_1$ with
$$f_0:=\sum_{e(f)\leq e(i)<e_0} c_i x^i\text{ and } f_1:=\sum_{e(n)\geq e_0} c_n x^n.$$
Then $\mathrm{mt}(f_0)=q(f)$ and $e(\mathrm{mt}(f_0))=e(f_0)=0$ and hence $f_0\sim_r x^{q(f)}$ by Lemma \ref{lm2.1}. That is, $\varphi(f_0)=x^{q(f)}$ for some coordinate change $\varphi\in Aut(K[[x]])$. It yields that
$$g:=\varphi(f)=\varphi(f_0)+\varphi(f_1)=x^{q(f)}+\varphi(f_1).$$
By Remark \ref{rm2.2}, $d(g)=d(f)$ and
$$\varphi(f_1)=\sum_{e(n)\geq e_0}\lambda_n x^n$$
for some $\lambda_n\in K$. Hence
$$f\sim_r g\sim_r j^{d(g)}(g)=x^{\mathrm{mt}(f)}+\sum_{n\in \bar\Lambda(f)}\lambda_n x^n$$
due to Corollary \ref{coro2.1}.
\end{proof}
From Proposition \ref{pro2.1} and Remark \ref{rm2.2}(b), replacing $f$ by $\bar f$ if $e(f)>0$, and denoting $\Delta:=\mathrm{supp}(f)$ we can conclude that 
$$\sharp \Lambda(f)\leq \left\lfloor \frac{q}{p^{e_0}}\right\rfloor-\frac{m}{p^{e_0}}+2\leq \left\lfloor \frac{d}{p^{e_0}}\right\rfloor-\frac{m}{p^{e_0}}+2=\sharp\bar\Lambda(f).$$
The following theorem is therefore stronger than Lemma \ref{lm2.2}  because it reduces the number of parameters.
\begin{theorem}[Normal form of univariate power series]\label{thm2.1}
With $f$, $\mathrm{mt}(f)$ and $\Lambda(f)$ as above, we have
$$f\sim_r x^{\mathrm{mt}(f)}+\sum_{n\in \Lambda(f)} \lambda_{n} x^{n}$$
for suitable $\lambda_n\in K$.
\end{theorem}
\begin{proof}
We set $\Delta:=\mathrm{supp}(f)$ and use the notations as in Definition \ref{def2.1}. It is sufficient to prove the theorem for the case that $e(m)>e$, because the case $e(m)=e$ follows from Lemma \ref{lm2.1}. Then
$$\Lambda_0(\Delta)=\left\{ n\in \mathbb N \ |\ m<n\leq \bar d_0, e_0<e(n)\right\}\cup \{q_0\},$$
$$\Lambda_1(\Delta)=\left\{n\in \mathbb N\ |\ q \leq  n\leq d, e\leq e(n)<e_0\right\}.$$
We decompose $f=f_0+f_1$ with
$$f_0:=\sum_{i<q} c_i x^i\text{ and } f_1:=\sum_{n\geq q} c_n x^n.$$
Applying Lemma \ref{lm2.2} to $f_0$ we obtain, by denoting $\Lambda'_0(\Delta):=\Lambda(\Delta)\cap\{n\in \mathbb N\ |\ n<q\}$ that 
$$f_0\sim_r x^{m}+\sum_{n\in \bar\Lambda(\Delta_0)}b_n x^n=x^{m}+\sum_{n\in \Lambda'_0(\Delta)}b_n x^n \mod x^{q},$$
for suitable $\lambda_n\in K$, since
$$\bar\Lambda(\Delta_0)\cap \{n\in \mathbb N\ |\ n<q\}\subset \Lambda'_0(\Delta).$$
This means that there exists a coordinate change $\varphi$ such that 
$$\varphi(f_0)=x^{m}+\sum_{n\in \Lambda'_0(\Delta)}b_n x^n \mod x^{q}.$$
We denote $g:=\varphi(f)$,
$$g_0:=x^{m}+\sum_{n\in \Lambda'_0(\Delta)}b_n x^n,$$
and 
$$g_1:=g-g_0:=\sum_{n\geq q} b_n x^n,\ b_q\neq 0.$$
We will construct a series $h$ such that $f\sim_r h$ and
$$h=x^m+\sum_{n\in \Lambda(\Delta)} \lambda_n x^n \mod x^d$$ by eliminating inductively all terms of exponent in  
$$I:=\{i\in \mathbb N\ |\ q\leq i\leq d, e\leq e(i)\}\setminus\Lambda(\Delta).$$
If we succeed then by Corollary \ref{coro2.1}
$$f\sim_r h\sim_r j^dh\sim_r x^m+\sum_{n\in \Lambda(\Delta)} \lambda_n x^n.$$
Let $i_1$ be the minimum exponent in $I$ for which $b_{i_1}\neq 0$. According to Remark \ref{rm2.1} the coordinate change
$$\varphi_1(x)=x+u_{l+1}x^{l+1}$$
with $l:=\dfrac{i_1-q_0}{p^{e_0}}$ and $u_{l+1}$ a root of the non-constant polynomial:
$$\sum_{n+lp^{e(n)}=i_1}b_{n}\big(n/{p^{e(n)}}\big)^{p^{e(n)}} X^{p^{e(n)}}+b_{i_1}=0,$$
makes the coefficient of $x^{i_1}$ vanish, and no term of exponent $i$ in $I$ with $i<i_1$ occurs. We prove the last claim by contradiction. Suppose the claim were false, then we could find $j\in I, j<i_1$ such that the coefficient of $x^j$ in $\varphi_1(g)$ differs from zero. That is, $j$ is an exponent of a term in $(x+u_{l+1}x^{l+1})^n$ for some $n\in \Lambda(\Delta)$ with $b_n\neq 0$. Then there exists an $i\in \mathbb N$ such that
$$j=n+ilp^{e(n)}.$$
Note that $i>0$ by the definition of $i_1$. This implies that 
$$n+ilp^{e(n)}\geq n+lp^{e(n)}>j \text{ for all } n\in \Lambda(\Delta)\text{ with }b_n\neq 0,$$
because
\begin{itemize}
\item if $e(n)\leq e_0$ then $n$ is either $q$ or $q_0$, and hence
$$q_0+lp^{e_0}=i_1 >j$$
and
$$q+lp^{e}\geq q_0+lp^{e_0}=i_1 >j$$
since $l\leq k$ due to Remark \ref{rm2.1}(e).
\item If $e(n)> e_0$ then $e(j)\geq e(n)>e_0$ and therefore $j>\bar d_0$. This implies that $\bar d_0=d_0<j<i_1<d$ and therefore 
$$l=\frac{i_1-q_0}{p^{e_0}}\geq k(\Delta_0).$$
It follows that 
$$n+ilp^{e(n)}\geq n+lp^{e(n)}\geq q_0+lp^{e_0}=i_1 >j.$$
\end{itemize}
This contradiction shows that there is no term of exponent $i$ in $I$ with $i<i_1$ in $\varphi_1(g)$. Hence we obtain by induction a series $h$ as required.
\end{proof}
Note that the families over $\Lambda(f)$ resp. $\bar\Lambda(f)$ in Theorem \ref{thm2.1} resp. Lemma \ref{lm2.2} contain all possible normal forms having the same set $\Lambda$ resp. $\bar\Lambda$ (and hence having the same $m, q, k$ and $d$). The number of parameters of normal forms in the
$\mu$--constant stratum (proof of Theorem \ref{thm3.1}) could be bigger.

The following example shows that this normal form is in general not the best one we can get. This means that, we can sometimes reduce the number of parameters even more.
\begin{example}\label{ex2.2}{\rm 
We consider 
$$f=x^8+x^{36}+x^{37}+\text{terms of higher order}$$
in characteristic $2$, as in Example \ref{ex2.1}. Then $d(f)=41$ and
$$\Lambda(f)=\{16,24,32,36,37,38,39,40,41\}.$$
It follows from Theorem \ref{thm2.1} that
$$f\sim_r x^8+\lambda_1x^{16}+\lambda_2x^{24}+\lambda_3x^{32}+\lambda_4x^{36}+\lambda_5x^{37}+\lambda_6x^{38}+\lambda_7x^{39}+\lambda_8x^{40}+\lambda_9x^{41}$$
for suitable $\lambda_i\in K$. 

On the other hand, applying Lemma \ref{lm2.1} to $f_1:=f-(x^8+x^{36})$ we get $f_1\sim_r x^{37}$. That is, $\varphi(f_1)=x^{37}$ for some coordinate change $\varphi$. It yields
$$\varphi(f)=a_0x^8+a_1x^{16}+a_2x^{24}+a_3x^{32}+a_4x^{36}+x^{37}\mod x^{41}.$$ 
By Proposition \ref{pro2.2},
$$f\sim_r \varphi(f)\sim_r a_0x^8+a_1x^{16}+a_2x^{24}+a_3x^{32}+a_4x^{36}+x^{37}+a_5x^{40}$$
and hence
$$f\sim_r x^8+b_1x^{16}+b_2x^{24}+b_3x^{32}+b_4x^{36}+b_5x^{37}+b_6x^{40}.$$
This shows that, we can find a ``better normal form'' for $f$. Moreover by the coordinate change
$$x+b_6/b_5x^4,$$
we can even get rid of the term $b_6x^{40}$ and obtain that
$$f\sim_r x^8+c_1x^{16}+c_2x^{24}+c_3x^{32}+c_4x^{36}+c_5x^{37}.$$
}\end{example}
In the following, we will give a set of terms of $f$ which can not be removed by coordinate changes and then we conjecture the ``best normal form'' for $f$.
\begin{remark}\label{rm2.3}{\rm 
Let $f\in K[[x]]$ be such that $\mu(f)<\infty$. Let $\Delta:=\mathrm{supp}(f)$ and let
$$q_{i}:=\min\{n\in \Delta\ |\ e(n)\leq i\}.$$
Then 
$$q(f)=q_0\geq q_1\geq\ldots\geq q_{e(m)}=m=q_i, \text{ for all } i\geq e(m).$$
We can see easily that the set $\{q_0,\ldots,q_{e(m)}\}$ is the set of exponents of terms which can not be removed by coordinate changes. However it is not true in general that
$$f\sim_r \sum_{i=1}^{e(m)}\lambda_i x^{q_i}$$
for suitable $\lambda_i\in K$ as the following example shows:
$$f=x^8+x^{36}+x^{37}+x^{38}\in K[[x]] \text{ with } \mathrm{char}(K)=2.$$
Then 
$$q_0=q_1=q=37, q_2=36, q_3=m=8.$$
It is not difficult to see that 
$$f\not\sim_r \lambda_0x^8+\lambda_1 x^{36}+\lambda_2 x^{37}$$
for any $\lambda_0,\lambda_1,\lambda_2\in K$.
}\end{remark}
We like to pose the following conjecture. 
\begin{conjecture}\label{conj2.1}
With notations as in Remark \ref{rm2.3}, let $\Lambda^{*}(f):=\emptyset$ if $e(m)=0$, otherwise
$$\Lambda^{*}(f):=\{n\in \mathbb N\ |\ m<n\leq q, e(n)\geq i \text{ if } q_i\leq n <q_{i-1}\}.$$
Then $f$ is right equivalent to
$$x^{\mathrm{mt}(f)}+\sum_{n\in \Lambda^{*}(f)}\lambda_{n}x^n$$
for suitable $\lambda_n\in K$, and moreover this is a modular family. That is, for each $\lambda=(\lambda_n)_{n\in \Lambda^{*}(f)}$, there are only finitely many $\lambda'=(\lambda'_n)_{n\in \Lambda^{*}(f)}$ such that
$$x^{\mathrm{mt}(f)}+\sum_{n\in \Lambda^{*}(f)}\lambda_{n}x^n\sim_r x^{\mathrm{mt}(f)}+\sum_{n\in \Lambda^{*}(f)}\lambda'_{n}x^n.$$
\end{conjecture}
\section{Right modality}
\begin{theorem}\label{thm3.1}
Let $\mathrm{char } K=p > 0$.  Let $f\in \langle x\rangle \subset K[[x]]$ be a univariate power series such that its Milnor number $\mu:=\mu(f)$ is finite. Then
$$\mathcal{R}\text{-}\mathrm{mod}(f)=\left\lfloor\mu/p\right\rfloor.$$
\end{theorem}
For the proof we need the following lemmas which are proven in \cite{GN13} for unfoldings but the proof works in general (for algebraic families of power series).

Let us recall the notion of unfoldings (see, \cite{GN13}). Let $T$ be an affine variety over $K$ with the structure sheaf $\mathcal{O}$ and its algebra of global section $\mathcal{O}(T)$. An element $f_{t}({x}):=F(x,{t})\in \mathcal{O}(T)[[x]]$ is called an {\em algebraic family of power series} over $T$. A family $f_{t}({x})$ is said to be {\em modular} if for each $t\in T$ there are only finitely many $t'\in T$ such that $f_{t'}$ is right equivalent to $f_{t}$. An {\em unfolding}, or {\em deformation with trivial section} of a power series $f$ at ${t}_0\in T$ over $T$ is a family $f_{t}(x)$ satisfying $f_{{t}_0}=f$ and $f_{t}\in \langle {x}\rangle $ for all ${t}\in T$. 
\begin{remark}\label{rm3.1}{\rm
Let $f\in\langle x\rangle \subset K[[x]]$ be a univariate power series with Milnor number $\mu<\infty$. Then the system $\{x,x^2,\ldots,x^{\mu}\}$ is a basis of the algebra $\langle x\rangle/\langle x\cdot\tfrac{\partial f}{\partial x}\rangle$. By \cite[Prop. 2.14]{GN13} the unfolding over $\mathbb A^{\mu}$,
$$f_{t}(x):=f+\sum_{i=1}^{\mu} t_i\cdot x^i$$
with $t:=(t_1,\ldots,t_{\mu})$ the coordinates of $t\in \mathbb A^{\mu}$, is an algebraic representative of the semiuniversal deformation with trivial section of $f$.
}\end{remark}
\begin{lemma}\label{lm3.3}
With $f$ and $f_{t}({x})$ as in Remark \ref{rm3.1}, assume that there exists a finite number of algebraic families of power series $h^{(i)}_{t}({ x})$ over varieties $T^{(i)},i\in I$ and an open subset $U\subset \mathbb A^{\mu}$ satisfying: for all ${t}\in U$ there exists an $i\in I$ and ${t}_i\in T^{(i)}$ such that $f_{ t}({ x})$ is right equivalent to $h^{(i)}_{{t}_i}({ x})$. Then 
$$\mathcal{R}\text{-}\mathrm{mod}(f)\leq \max_{i=1,\ldots,l} \dim T^{(i)}.$$
\end{lemma}
\begin{proof}
cf. \cite[Proposition 2.15(i)]{GN13}.
\end{proof}
\begin{lemma}\label{lm3.4}
If $f_t(x)$ is a modular unfolding of $f$ over $T$ then  
$$\mathcal{R}\text{-}\mathrm{mod}(f)\geq \dim T.$$
\end{lemma}
\begin{proof}
It follows from \cite[Propositions 2.12(ii) and 2.15(ii)]{GN13}.
\end{proof}
\begin{proof}[Proof of Theorem \ref{thm3.1}]
We first prove the inequality $\mathcal{R}\text{-}\mathrm{mod}(f)\leq\left\lfloor \mu/p\right\rfloor$. Indeed, let
$$I:=\{\Delta\subset \{1,\ldots,q(f)\} |\ q(f)\in \Delta\},$$
and let  
$$h_{{s}_\Delta}(x):=x^{m(\Delta)}+\sum_{n\in \Lambda(\Delta)} s_\Delta^{(n)} x^{n},\ \Delta\in I$$
the finite set of families over $A_{\Delta}\equiv \mathbb{A}^{l_{\Delta}}$ with $l_{\Delta}=\sharp \Lambda(\Delta)$ and $s_\Delta^{(n)},\ n\in \Lambda(\Delta)$ the coordinates of $s_\Delta$ in $A_{\Delta}$.

Notice that if $\Delta \in I$, then $e(\Delta)=0, q(\Delta)\leq q(f)$ and therefore, by Proposition \ref{pro2.1}, 
$$\dim A_{\Delta}=\sharp \Lambda(\Delta)\leq \left\lfloor q(\Delta)/p\right\rfloor\leq \left\lfloor q(f)/p\right\rfloor=\left\lfloor \mu/p\right\rfloor.$$
With $f_t$ as in Remark \ref{rm3.1}, setting
$$\Delta_t:=\{n\in \mathrm{supp}(f_t)\ |\ n\leq q(f)\}$$
for each $t\in \mathbb A^\mu$, we conclude that $\Delta_t\in I$ and $\Lambda(\Delta_t)=\Lambda(\mathrm{supp}(f_t))$ according to Remark \ref{rm2.1}(b). By Theorem \ref{thm2.1}, $f_t\sim_r h_{{s}_{\Delta_t}}$ for some ${s}_{\Delta_t}$. This implies that the finite set of families $h_{{s}_\Delta}(x), \Delta\in I$ satisfies the assumption of Lemma \ref{lm3.3}. Hence 
$$\mathcal{R}\text{-}\mathrm{mod}(f)\leq \max_{\Delta \in I} \dim A_{\Delta}\leq \left\lfloor \mu/p\right\rfloor.$$ 
In oder to prove the other inequality we consider the two following cases.\\ 
{\bf Case 1:} $m(f)=p$. 

Then $q:=q(f)=\mu(f)+1, k:=k(f)=\left\lfloor\frac{q-p}{p-1}\right\rfloor, d:=d(f)=q+k-1$ and $$\Lambda(f)=\{n\in \mathbb N\ |\ q \leq n\leq d, e(n)=0\}$$
and $\sharp \Lambda(f)=\left\lfloor q/p\right\rfloor$ due to Proposition \ref{pro2.1}. It follows from Theorem \ref{thm2.1} that
$$f\sim_r g:=x^p+\sum_{n\in \Lambda(f)} c_{n} x^{n}$$
for suitable $c_n\in K$ with $c_q\neq 0$. Consider the unfolding 
$$g_{\lambda}:=g+\sum_{n\in \Lambda(f)} \lambda_{n} x^{n}$$
of $g$ over $S:=\left\{\lambda=(\lambda_n)_{n\in \Lambda(f)}\in \mathbb A^{\sharp \Lambda(f)}\ |\ \lambda_q+c_q\neq 0\right\}$, where $\lambda_n, {n\in \Lambda(f)}$ are the coordinates of $\lambda$. Let us show that $g_{\lambda}$ is a modular unfolding. In fact, if $\lambda'=(\lambda'_n)_{n\in \Lambda(f)}\in S$ for which $g_{\lambda} \sim_r g_{\lambda'}$, then there exists a coordinate change
$$\varphi:=a x+a_{l}x^{l+1}+\ldots $$ such that 
$$\varphi(g_{\lambda}) = g_{\lambda'}.$$
Looking at the coefficient of $x^p$ we deduce that $a^p=1$ and therefore $a=1$. We have moreover that $l\geq k$, because if $l<k$, equivalently, $q+l>p(l+1)$ then $p(l+1)\in \mathrm{supp}(\varphi(g_{\lambda}))$ but $p(l+1)\not\in \mathrm{supp}(g_{\lambda'})$, that is $\varphi(g_{\lambda}) \neq g_{\lambda'}$, a contradiction. It then follows from Remark \ref{rm2.0}(d) that 
$$j^d(g_{\lambda})=j^d(\varphi(g_{\lambda}))=j^d(g_{\lambda'}),$$ i.e. $\lambda=\lambda'$. This implies that $g_{\lambda}$ is a modular unfolding and hence $$\mathcal{R}\text{-}\mathrm{mod}(f)=\mathcal{R}\text{-}\mathrm{mod}(g)\geq \sharp \Lambda(f)=\left\lfloor q/p\right\rfloor=\left\lfloor \mu/p\right\rfloor$$
due to Lemma \ref{lm3.4}
\\
{\bf Case 2:} $m(f)>p$. 

By the upper semicontinuity of the right modality (cf. \cite[Prop. 2.7]{GN13}) one has $$\mathcal{R}\text{-}\mathrm{mod}(f)\geq \mathcal{R}\text{-}\mathrm{mod}(f_s)$$ with $f_s=f+s\cdot x^p$, for all $s$ in some neighbourhood $W$ of 0 in $\mathbb A^1$. Take a $s_0\in W\setminus \{0\}$ then $\mathcal{R}\text{-}\mathrm{mod}(f_{s_0})=\left\lfloor \mu/p\right\rfloor$ by the first case and hence $$\mathcal{R}\text{-}\mathrm{mod}(f)\geq \mathcal{R}\text{-}\mathrm{mod}(f_{s_0})=\left\lfloor \mu/p\right\rfloor.$$
\end{proof}
\begin{remark}{\rm
We have $\mathcal{R}\text{-}\mathrm{mod}(f)\geq \sharp \Lambda(f)$ by Theorem \ref{thm3.1} and Proposition \ref{pro2.1} with equality if $m(f)\leq p$. Moreover, if $m(f)=p$, then $f_{\lambda}\sim_r f_{\lambda'}$ for $\lambda,\lambda'\in \Lambda(f)$ implies $\lambda=\lambda'$, which follows from the proof of Theorem \ref{thm3.1}. 

The example $f=x^{p+1}$ with $\mathcal{R}\text{-}\mathrm{mod}(f)=1$ but $\Lambda(f)=\emptyset$ shows that a strict inequality $\mathcal{R}\text{-}\mathrm{mod}(f)> \sharp \Lambda(f)$ can happen.
}\end{remark}
With $f$ and the semiuniversal unfolding $f_{t}(x)$ as in Remark \ref{rm3.1} we define
$$\Delta_{\mu}:=\{{t}\in \mathbb A^{\mu}\ |\ \mu(f_{t})=\mu\}$$
the {\em $\mu$-constant stratum} of the unfolding $f_{t}$.
\begin{corollary}\label{coro3.1}
Let $f\in\langle x\rangle \subset K[[x]]$ with the Milnor number $\mu<\infty$. Then 
$$\mathcal{R}\text{-}\mathrm{mod}(f)=\dim\Delta_\mu.$$
\end{corollary}
\begin{proof}
For each ${t}=(t_1,\ldots,t_\mu)\in \mathbb A^{\mu}$, if the set $N_{t}:=\{i=1,\ldots, \mu\ |\ t_i\neq 0, e(i)=0\}$ is not empty, then $\mu(f_{t})=n-1<\mu$ with $n:=\min \{i\ |\ i\in N_{t}\}$. This implies that 
$$\Delta_{\mu}=\{ t=(t_1,\ldots,t_\mu)\in \mathbb A^{\mu}\ |\ t_i=0\ \mathrm{if}\ e(i)=0\}.$$
It yields that
$$\dim \Delta_{\mu}=\sharp\left\{1\leq n\leq \mu\ |\ e(n)>0\right\}=\left\lfloor {\mu}/{p}\right\rfloor$$
and hence $\mathcal{R}\text{-}\mathrm{mod}(f)=\dim\Delta_\mu$ by Theorem \ref{thm3.1}.
\end{proof}



\addcontentsline{toc}{section}{Bibliography}
\begin{thebibliography}{99}
\baselineskip=16pt


\bibitem[Arn72]{Arn72} Arnol'd V. I., {\em Normal forms for functions near degenerate critical points, the Weyl groups of $A_k, D_k, E_k$ and Lagrangian singularities,} Functional Anal. Appl. 6 (1972) 254-272.

\bibitem[BGM12]{BGM12} Boubakri Y., Greuel G.-M., and Markwig T., {\em Invariants of hypersurface singularities in positive characteristic,} Rev. Mat. Complut. 25 (2012), 61-85.




\bibitem[GN13]{GN13} Greuel G.-M., Nguyen H. D., {\em Right simple singularities in positive characteristic,} to appear in the Journal f\"ur die Reine und Angewandte Mathematik (2013).

\url{http://arxiv.org/abs/1206.3742}






\bibitem[Ng13]{Ng13} Nguyen H. D., {\em Classification of singularities in positive characteristic,} Ph.D. thesis, TU Kaiserslautern, Dr. Hut-Verlag (2013). 

\url{http://www.dr.hut-verlag.de/9783843911030.html}






\end{thebibliography}
\end{document}